\documentclass[10pt]{amsart}
\usepackage{amsthm}

\newtheorem{theorem}{Theorem}[section]

\newtheorem{lemma}[theorem]{Lemma}

\theoremstyle{definition}
\newtheorem{definition}[theorem]{Definition}

\theoremstyle{remark}
\newtheorem{remark}[theorem]{Remark}

\numberwithin{equation}{section}

\newcommand{\eq}[1][r]
   {\ar@<-3pt>@{-}[#1]
    \ar@<-1pt>@{}[#1]|<{}="gauche"
    \ar@<+0pt>@{}[#1]|-{}="milieu"
    \ar@<+1pt>@{}[#1]|>{}="droite"
    \ar@/^2pt/@{-}"gauche";"milieu"
    \ar@/_2pt/@{-}"milieu";"droite"}

\newcommand{\Apre}{{R}^{\text{pre}}}
\newcommand{\F}{\mathcal{F}}
\newcommand{\OO}{\mathcal{O}}

\DeclareMathOperator{\Spec}{Spec}
\usepackage{xcolor}

\input xy
\xyoption{all}
\usepackage[all]{xy}
\usepackage{hyperref}
\usepackage{amsfonts}
\usepackage{amsmath}
\usepackage{amsthm}
\usepackage{amssymb}
\usepackage{mathrsfs}

\begin{document}

\title[]{Geometric characterization of rings with Krull dimension $\leq 1$}

\author{Jes\'us Mart\'in Ovejero}

\address{Lloyds Banking Group, 10 Gresham Street, EC2V 7AE London}
%\curraddr{Department of Mathematics, University of Le\'on, Spain}
\email{}
%\thanks{The first author was supported in part by ......}

%    General info
\subjclass{ 14R99}

\date{}

%\dedicatory{This paper is dedicated to ......}

\keywords{Affine schemes, algebraic geometry}

\begin{abstract}
	In this paper we provide a new characterization of noetherian rings with Krull dimension $\leq 1$ in terms of its spectrum. 
\end{abstract}

\maketitle

\section{Introduction}
The aim of this short paper is to prove the following theorem 
\begin{theorem}Let $R$ be a commutative noetherian ring. Then  $\dim(R)\leq 1$ if and only if every open set of $X=\Spec(R)$ is affine (when $\dim(R)=1$ we assume in the converse that $R$ is reduced).\end{theorem}

\section{Characterization of rings with Krull dimension $\leq 1$}
Before proving the main result of this paper, let us introduce some notations, definitions, and previous results.  Let $R$ be a commutative noetherian ring with unit, and let $\Spec(R)$  be the set of prime ideals of $R$  endowed with the Zariski topology.
\begin{definition}The elements  of $R$ will be called functions, while the elements of $\Spec(R)$ will be called points.\end{definition}
\begin{definition}A function  $f\in R$  vanishes at a point  $x\in\Spec(R)$  if $f\in\mathfrak{p}_{x}$, i.e., if the class of  $f$ in $R/\mathfrak{p}_{x}$ is $0$.\end{definition}
Let $X=\Spec(R)$ be the spectrum of $R$. The structure of scheme on $X$ is given by  considering on it the structural sheaf, which is sheaf associated to the presheaf that assigns to each open set $U\subseteq X$ the ring of non-vanishing functions on $U$, that is:
$$\Apre(U)=R_{S}\text{,}\hspace{0.3cm}\text{with}\hspace{0.3cm} S=\big\{\text{Multiplicative set of non vanishing functions on U}\big\}$$
The open sets of the form  $$U_{f}:=\Spec(R_{f})\text{,}$$ where $R_{f}$ denotes the localization of $R$ by the multiplicative system defined by $f$, are know as affine open sets. Those open sets are a basis for the Zariski topology.
\begin{definition}A topological space is a Zariski topological space if it is noetherian and sober.\end{definition}
\begin{remark}The spectrum of a ring is a Zariski topological space. The Hilbert's nullstellensatz  exhibits a bijective correspondence between  irreducible closed sets and  prime radical ideals. In particular, the point defined by a prime ideal in the spectrum of a ring is precisely the generic point of the closed set defined by the set of zeros of the ideal. It is clear now that if an irreducible closed has two generic points, then, those two points, must be equal.\end{remark}

\begin{lemma}\label{lema1} Let $X$ be a Zariski topological space, $p\in X$ a closed point and $X_{p}$ the subset of $X$ formed by all the points $q\in X$ such that $p\in \overline{\{q\}}$. Let us consider in   $X_{p}$  the induced topology by $X$, and let $j:X_{p}\hookrightarrow X$ be the canonical inclusion. For every sheaf $\F$, and for each $i\in\mathbb{N}_0$,  the following holds
	$${\rm{H}}^{i}_{p}(X,\F)= {\rm{H}}^{i}_{p}(X_{p},\F_{p}:=j^{\ast}\F)$$
\end{lemma}
\begin{proof}By  construction, it is enough to show that
	$\Gamma_{p}(X,\F)\simeq \Gamma_{p}(X_{p},\F_{p})$. Let us consider the canonical ring homomorphism
	$$\Gamma(X,\F)\rightarrow \Gamma(X_{p},\F_{p})=\varprojlim_{p\in U}\F(U)\text{.}$$
	The above homomorphism induces another one in the localization
	$$f:\Gamma_{p}(X,\F)\rightarrow\Gamma_{p}(X_{p},\F_{p})\text{,}$$
	so it is enough to prove that $f$ is  bijective. \\\
	
	If $f(t)=f(s)\in \Gamma_{p}(X_{p},\F_{p})$, then, $t_{p}=s_{p}$, because $f$ maps each section to its germ at the point $p$. This is true for each point $p$, so $s=t$ and therefore $f$ is injective.\\\
	
	We claim that $f$ is surjective. Let $s\in \Gamma_{p}(X_{p},\F_{p})=\F_{p}$ be a section.  It is possible to find a open neighbourhood  $p\in U$ and a section $s_{U}\in \F(U)$ such that  $s_{U}$  is a representative of the section $s$. Without loss of generality, we can assume that the support of $s_{U}$ is $p$. Setting $V:=X-\{p\}$ we get $(s_{U})_{\mid U\cap V}=0$, so extending   $s_{U}$ by the zero section we obtain a global section with support $p$. We conclude that $f$ is surjective.  \end{proof}

\begin{lemma}\label{lemacrucial}Let $R$ be a noetherian reduced commutative ring with  ${\rm{dim}}(R)\leqslant 1$. Every open set of  $X=\Spec(R)$ is the complement of a finite number of closed points.\end{lemma}
\begin{proof} Let  $U$ be  an open set of $X$ and let   $Z=X-U$ be its complement in $X$. By hypothesis, $X$ is a noetherian topological space, therefore, $Z$ can be expressed as the finite union of its irreducible components, being, each one of them, closed sets in $X$. It is enough to show that if $Z$ is a closed irreducible set of $X$ then, $Z$, is a closed point. By hypothesis $X$ is reduced, so $X$ is affine if and only if its irreducible components are affine  (\cite[Ex III. 3.2]{har}). We can assume that $X$ is irreducible and   $Z=X-U$. The dimension of  $U$ coincide with the dimension of $X$ and therefore $Z$ is irreducible and $0$-dimensional, thus $Z$ is discrete and in particular it is a closed point. 
\end{proof}
Since the Serre's criterion for affineness will be used in the proof of the main Theorem, we recall the statement.

\begin{theorem}\label{serre}[Serre's criterion for affineness] Let $X$ be a noteherian scheme. Then the following conditions are equivalent
	\begin{enumerate}
		\item $X$ is affine;
		\item ${\rm{H}^{i}}(X,\F)=0$ for all quasi-coherent sheaf  $\F$ and for all $i>0$;
		\item ${\rm{H}^{i}}(X,\mathcal{I})=0$ for all coherent sheaf of ideals $\mathcal{I}$ and for all $i>0$.
	\end{enumerate}
\end{theorem}

\begin{theorem} Let  $R$ be a commutative noetherian ring, then $\dim(R)\leq 1$ if and only if every open set of $X=\Spec(R)$ is affine  (when $\dim(R)=1$  we assume in the converse that $R$ is reduced).\end{theorem}
\begin{proof}Let us suppose that every open set of $X$ is affine and let
	$\mathcal{O}_{X}$ denotes the structural sheaf. Let us consider a closed point  $p\in X$ and its complement  $U=X-\{p\}$. For every quasi-coherent  $\mathcal{O}_{X}$-module  $\F$ we have the local cohomology exact sequence
	$$\cdots\rightarrow \text{H}^{i}_{p}(X,\F)\rightarrow \text{H}^{i}(X,\F)\rightarrow\text{H}^{i}(U,\F)\rightarrow \text{H}^{i+1}_{p}(X,\F)\rightarrow\cdots$$
	The sets $U$ and $X$ are affine open sets, so,   by the Serre's criterion \ref{serre}, $\F$ is acyclic on them, i.e.   $\text{H}^{i}(X,\F)=\text{H}^{i}(U,\F)=0 \hspace{0.1cm}\forall i>0$, and therefore $$\text{H}^{i}_{p}(X,\F)=0\hspace{0.5cm}\forall i>1\text{.}$$ This remark and the Lemma \ref{lema1} allow us to conclude that $$\text{H}^{i}_{p}(X,\F)=\text{H}^{i}_{p}(X_{p},\F_{p})=\text{H}^{i}_{p}(R_{p},\F_{p})= 0\hspace{0.2cm}\forall i>1$$
	In particular,  if  $\F_{p}$ is a finitely generated $R_{p}$-module, then, $\text{dim}(F_{p})\leq 1$ \break(see \cite[Th 6.1.2]{eee}), because $R_{p}$ is a noetherian local ring of maximal ideal  $p$. If we consider $\F=\mathcal{O}_{X}$, then  $\F_{p}=R_{p}$, and therefore the dimension of  $R_{p}$ is less or equal than 1 for each closed point, so  $\text{dim}(R)\leq 1$.\\\

	Reciprocally, let us suppose that $\text{dim}(R)=0$. In this situation, $R$ is an artinian ring and therefore its spectrum is a finite number of points equipped with the discrete topology. This implies that every open set is the complement of a finite number of closed points. If $\text{dim}(R)=1$ and $R$ is reduced, by  Lemma \ref{lemacrucial},   every open set  in $X=\Spec(R)$ is the complement of a finite number of closed points $\{p_{1},...,p_{r}\}$. In both cases, every open set $U$ in $X$ is the complement of a finite number of closed point $\{p_{1},...,p_{r}\}$.\\\
	
	Let us suppose that $U=X-\{p\}$ is the complement of a single closed point. Since $\text{dim}(R)\leq 1$, by the Grothendieck's vanishing theorem  $$\text{H}^{i}_{p}(R_{p},M)=0 \quad \forall i>1\geq \text{dim}(M)$$  for all  $R_{p}$-module finitely generated    $M$. Since every $R$-module is the direct limit of its finitely generated modules, and the cohomology commutes with direct limits, we conclude that for every quasi-coherent $\OO_{X}$-module $\F$ it holds that $\text{H}^{i}_{p}(X,\F)=0$ for all $i>1$.  On the other hand, since $X$ is affine, by the Serre's criterion for affineness,  every quasi-coherent sheaf of  $\mathcal{O}_{X}$-modules  is acyclic, and therefore, $\text{H}^{i}(X,\F)=0$ for all $i>0$. Using the  local cohomology exact sequence  we obtain that $\text{H}^{i}(U,\F_{\mid U})=0$ for all $i>0$. Now, since every quasi-coherent sheaf on $U$ is the restriction of a quasi-coherent sheaf on  $X$, by the Serre's criterion for affineness, $U$ is affine.\\\
	
	Our  induction hypothesis will be that if $U=X-\{p_{1},...,p_{r-1}\}$ is the complement of $r-1$ closed points, then $U$ is affine. Furthermore, by the local cohomology exact sequence,  our induction hypothesis implies that for each collection $Y=\{q_{1},...,q_{r-1}\}$ of $r-1$ closed points, and for each quasi-coherent $\OO_{X}$-module $\F$, the following holds
	\begin{equation}\label{inductiontwo}\text{H}^{i}_{Y}(X,\F)=0\hspace{0.5cm}\forall i>1\end{equation}
	Let us suppose that $U$ is the complement of $r$-closed points $\{p_{1},...,p_{r}\}$. We claim that $U$ is affine. Firstly, we are going to prove that if $Y=\{p_{1},...,p_{r}\}$, then, $$\text{H}^{i}_{Y}(X,\F)=0  \hspace{0.5cm}\forall i>1$$   for all quasi-coherent $\OO_{X}$-module $\F$. Let us use the following notation $$Y_{1}:=\{p_{1},...,p_{r-1}\}\quad \quad Y_{2}=\{p_{r}\}\text{.}$$ By the  Mayer-Vietoris exact sequence (see \cite[Ex. III, 2.4]{har})
	$$\cdots\rightarrow  \text{H}^{i}_{Y_{1}\cap Y_{2}}(X,\F)\rightarrow\text{H}^{i}_{Y_{1}}(X,\F)\oplus\text{H}^{i}_{Y_{2}}(X,\F)\rightarrow\text{H}^{i}_{Y_{1}\cup Y_{2}=Y}(X,\F)\rightarrow \text{H}^{i+1}_{Y_{1}\cap Y_{2}}(X,\F)\rightarrow\cdots$$ 
	and  by \eqref{inductiontwo}, we conclude that 
	\begin{equation}\label{aci}\text{H}^{i}_{Y}(X,\F)=0\hspace{0.5cm}\forall i>1\text{,}\end{equation}
	for all quasi-coherent $\OO_{X}$-module $\F$. For every quasi-coherent $\OO_{X}$-module $\F$ we have the local cohomology exact sequence
	$$\cdots\rightarrow \text{H}^{i}_{Y}(X,\F)\rightarrow \text{H}^{i}(X,\F)\rightarrow\text{H}^{i}(U,\F)\rightarrow \text{H}^{i+1}_{Y}(X,\F)\rightarrow\cdots$$
	Again, by the Serre's criterion for affineness
	$$\text{H}^{i}(X,\F)=0\hspace{0.5cm}\forall i>0$$
	and by $\eqref{inductiontwo}$
	$$\text{H}^{i}_{Y}(X,\F)=0\hspace{0.5cm}\forall i>1$$
	thus $$\text{H}^{i}(U,\F)=0 \hspace{0.5cm}\forall i\geq1$$
	Since every quasi-coherent sheaf on $U$ is the restriction of a quasi-coherent sheaf on $X$, by the Serre's criterion for affineness, we conclude that $U$ is affine.
\end{proof}

\end{document}